\documentclass[a4paper,9p]{article}

\usepackage{amsmath}
\usepackage{pstricks,pst-node,pst-tree}

\usepackage{amsfonts}
\usepackage{amssymb}
\usepackage{amsthm}
\usepackage{verbatim}
\usepackage[active]{srcltx}
\usepackage[english]{babel}
\usepackage{fontenc}
\usepackage{color}
\usepackage{cancel}
\usepackage{graphicx}
\usepackage{srcltx}

 \newtheorem{dfn}{Definition}
 \newtheorem{prp}{Proposition}
\newtheorem{remark}{Remark}
 \newtheorem{notation}{Notation}
\newtheorem{cor}{Corollary}
\newcommand{\R}{\mathbb{R}} 
\newcommand{\N}{\mathbb{N}}

\begin{document}
\title{A systematic approach for doing an \textit{a priori} identifiability study of dynamical nonlinear models.}
 
\maketitle

\begin{center}
N. Verdi\`ere$^1$ and S. Orange\footnote{Normandie Univ, UNIHAVRE, LMAH, FR-CNRS-3335, ISCN, 76600 Le Havre, France\\
(Corresponding authors: Nathalie.Verdiere@univ-lehavre.fr, Sebastien.Orange@univ-lehavre.fr)}\end{center}

\begin{abstract}
This paper presents a method for investigating, through an automatic procedure, the (lack of) identifiability of parametrized dynamical models. This method takes into account constraints on parameters and returns parameters whose estimations make the model identifiable. It is based on i) an equivalence between an extension of the notion of identifiability and the existence of solutions of algebraic systems, ii) the use of symbolic computations for testing their existence.

This method is described in details and is applied to two examples, of which the last one involves 12 parameters. 

\end{abstract}

\textbf{Keywords :} Relative identifiability; Nonlinear models; Algorithm


\pagestyle{myheadings}
\thispagestyle{plain}
\markboth{ }{ }

\section{Introduction}
Since several years, the study of identifiability of parametrized models generates a growing interest in many scientific domains \cite{Pironet2016,Walsh2016}.
Indeed, when a new model is developed, some parameters may not be directly accessible from the experiments. For evaluating them, some numerical procedures using real data have been developed \cite{VDJ,Tarantola2005}. The identifiability guarantees that from these real data correspond exactly one parameter vector. Thus, if the model is not identifiable, the output of the numerical procedure may not return the correct parameter vector.

Finally, when one wants to study this property, the following main questions arise naturally: 
\begin{enumerate}
\item Is the model identifiable? How initial conditions, and more generally constraints on parameters, can be taken into account in this question?
\item Is it possible to obtain the lists of key parameters whose estimation turns the unidentifiable model into an identifiable one? 
\end{enumerate}

For the first question, several methods can be found in the literature  (see~\cite{chapman,chappell,dj,lili:2004,DF,Pohj,Janzen2016,FG,ljung,Stigter2015,W}) and as it was noticed in \cite{dj,SAD,VDJ}, the identifiability of a model may depend on the consideration of initial conditions or more generally constraints on parameters. For the second question, there is no answer at our knowledge. \\

In order to answer to these two questions by the mean of an automatized procedure, our present works used as a first step the classical method based on the Rosenfeld-Groebner algorithm (see~\cite{boulier:tel-00137866,NOLCOS}). This algorithm permits to eliminate state variables of dynamical systems and returns relations linking outputs, inputs and unknown parameters of the initial system. These relations, called \textit{input-output polynomials} are the basis of the identifiability method developed in \cite{NOLCOS}. This method consists in checking the identifiability of the model through the study of the coefficients of the input-output polynomials. In the first version of this method, initial conditions were ignored and few developments were proposed to consider them (see  \cite{dj,SAD,VDJ}). 



The first contribution of this paper, presented in section \ref{section1}, is a theoretical one. Since the identifiability of a parameter depends on the identifiability or not of some other parameters, the classical identifiability definition have to be extended. For this purpose, we propose the notion of \textit{relative identifiability}, that is identifiability of a given parameter when some other parameters are supposed to be known. We then generalize a result established in \cite{dj} for classical identifiability. This result links output trajectories and solutions of  algebraic systems built on input-output polynomials. More precisely, we give an equivalence between the relative identifiability of a parameter and the emptiness of the solution set of equations and/or inequalities system. The latter can be tested with computer algebra softwares. This theoretical approach permits to obtain an automatized approach for studying the (relative) identifiability of systems.\\ 

The second contribution, presented in section \ref{section2}, is an algorithm based on this equivalence. It returns lists of parameters, each of these lists giving a set of key parameters in order to obtain the identifiability of the model. Our algorithm can take into account the definition domain of the parameters and, possibly, some algebraic relations, for example, those provided by initial conditions. This new tool can help an experimenter to elaborate strategies about parameters that should be estimated to obtain the identifiability of the model. Notice that the equivalence we establish in section~\ref{section1} ensures that the set of lists provided by our algorithm is exhaustive for the considered model.  
In section~\ref{secexemples}, our algorithm is applied to two examples. The first one revisits the identifiability of microbial growth in a batch reactor, the second one concerns an epidemiological model and involves 12 parameters.


\section{Relative identifiability}\label{section1}
\subsection{Class of systems}
All along this paper, we consider nonlinear dynamical parametrized models (controlled or uncontrolled) of the following form:
\begin{equation}\label{saresoudre}
\Gamma^{\Theta}
\left\{\begin{array}{l}
\dot{x}(t,\Theta )=g(x(t,\Theta  ),u(t),\Theta  ),\\
y(t,\Theta )=h(x(t,\Theta  ),u(t),\Theta  ), 
\end{array}\right.
\end{equation}
where: 
\begin{itemize}
\item $x(t,\Theta )\in \R^n$ denotes the state variables and $y(t,\Theta )\in \R^s$ the outputs respectively,
\item the functions $g$ and $h$ are real, rational and {analytic }on $M$\footnote{These assumptions are not restrictive since lots of models can be reduced to a rational and{ analytic} model by variable change (see.~\cite{SAB}).}, where $M$ is an open subset of $\R^n$ such that $x(t,\Theta )\in M$ for every $t\in [t_0,\ T]$. $T$ is a finite or infinite time bound,
\item $u(t)\in R^r$ is the control vector, 
\item the vector of parameters $\Theta =(\theta_1,\ldots,\theta_m)$ belongs to a subset $\mathcal{D}$ of $ \subseteq \R^{m}$ where  $\mathcal{D}$ is an, a priori, known set of admissible parameters.
\end{itemize}
When an initial condition $x_0\in \R^n$ is given, the supplementary equation $ \,x(0,\Theta)=x_0$ can be added to System~(\ref{saresoudre}).

\subsection{Admissible parameters and semialgebraic sets}\label{DefC}
In many models, it is natural to work only on a subset $\mathcal{C}$ of the set of admissible parameters $\mathcal{D}$ when considering additional constraints on the parameter vector such as initial conditions. This may be the case either to ensure the existence of solutions or because of the nature of the problem. For example, in trophic chains studies, number of preys eaten by predators can not be greater than the total population of preys and in mechanic, a mass can not be negative. 
 
From now on, we will suppose that these constraints on $\Theta$ can be formulated by the mean of algebraic equations and/or inequalities. This leads naturally to consider semialgebraic sets:

\begin{dfn} {(See~\cite{Basu:2006:ARA:1197095})} A set of real solutions of a finite set of polynomial equations and/or polynomial inequalities is called a semialgebraic set.
\end{dfn}
    
The introduction of semialgebraic sets enables us to use efficient computer algebra libraries developed for manipulating this kind of mathematical object like the Maple packages \verb Raglib  (See~\cite{RAGLib}) or \verb SemiAlgebraicSetTools (See also~\cite{qepcadb,DISCOVERED}, for example).

From now on, we suppose that:
\begin{itemize}
\item $C(\Theta)$ denotes the set of all algebraic equations and inequalities verified by the components of the parameter vector $\Theta \in \mathcal{D}$ of the model;
\item the set $\mathcal{C}$  is the semialgebraic set defined by $C(\Theta)$.
 \end{itemize}

\subsection{Relative identifiability}\label{identifiability}
 If initial conditions are not considered in (\ref{saresoudre}), authors in~\cite{ljung} have proposed the following definition of identifiability in the framework of algebra: 

\begin{dfn}\label{defident}
The model  $\Gamma^{\Theta} $ is said \textit{globally identifiable} if there exists an input $u$ such that the set of outputs generated by $\tilde \Theta \in \mathcal{D}$ has an empty intersection with those generated by $ \Theta $ when $\Theta \neq \tilde{\Theta}$. 
\end{dfn}
In the same way, a local definition had also been proposed in~\cite{ljung}: $\Gamma^{\Theta} $ is said \textit{locally identifiable} if the latter definition is verified in a neighborhood $\nu(\Theta) \subset \mathcal{D}$ of $\Theta$. In the case of known initial conditions, model $\Gamma^{\Theta}$ has a unique solution and the identifiability definition consists in verifying that, for two different parameter vectors $ \Theta$ and $\tilde \Theta$ in $\mathcal{D}$, systems $\Gamma^{ \Theta}$ and $\Gamma^{\tilde \Theta}$ yield different outputs. In the case of uncontrolled system, when $u$ is a null function, the previous definitions stay valid.

The next definitions extend the notions of global and local identifiability. They focus on the identifiability of one particular parameter when some other parameters, and eventually none, are supposed to be identifiable. These definitions appear naturally when one wants to determine which parameters should be determined in order to turn the model into an identifiable one.\\

\begin{dfn} \label{defidentrelat}
\begin{enumerate}
\item The parameter $ \theta_{a_{r+1}}$ is said to be relatively identifiable with respect to a set of parameters $\{ \theta_{a_1}, \ldots, \theta_{a_r}\}$ (resp. identifiable) if there exists an imput $u$ such that the set of outputs generated by $\tilde \Theta =(\tilde \theta_1,\ldots, \tilde \theta_m) \in \mathcal{D}$ has an empty intersection with those generated by $ \Theta $ when $\theta_{a_1}=\tilde \theta_{a_1},\, \ldots,\, \theta_{a_r}= \tilde \theta_{a_r} $ and $ \theta_{a_{r+1}} \neq \tilde\theta_{a_{r+1}}$ (resp. when $ \theta_{a_{r+1}} \neq \tilde\theta_{a_{r+1}}$).  

\item The relative identifiability study of system $\Gamma^{ \Theta}$ is the determination of the relative identifiability of any parameter relatively to any set of parameters taken among $\theta_{1}, \ldots, \theta_{m}$.\\
\end{enumerate}
\end{dfn}

\begin{remark}\label{remdefidentrelat}
\begin{itemize}
 \item These definitions can be extended to the local case by considering a neighbourhood $\nu(\Theta) \subset \mathcal{D}$ of $\Theta$ instead of $ \mathcal{D}$.
 \item Clearly, system $\Gamma^\Theta$ is globally identifiable if and only if all parameters $\theta_{i}$ ($1\leq i\leq m $) of $\Gamma^\Theta$ are identifiable. Consequently, identifiability of a system can be seen as a particular case of relatively identifiability.
\end{itemize}
\end{remark}

The following proposition is a direct consequence of Definition~\ref{defidentrelat}.

\begin{prp}\label{propdefidentrelat}
If $\theta_i$ is relatively identifiable with respect to an eventually empty subset $\mathcal{P} \subset \{ \theta_{1}, \ldots, \theta_{m}\}$ then $\theta_i$  is relatively identifiable with respect to any subset of $ \{ \theta_{1}, \ldots, \theta_{m}\}$ containing $\mathcal{P} $. 
\end{prp}

In the next section, we show how to reduce the relative identifiability study to a semialgebraic set problem.\\

\subsection{A differential algebra method to perform relative identifiability analysis}

The input-output approach consists from system (\ref{saresoudre}) and the Rosenfeld-Groebner algorithm to obtain specific differential polynomials of the following form

\begin{eqnarray}\label{eqpol}
P_i(y,u,\Theta )=m_{0,i}(y,u)+\sum_{k=1}^{n_i}c_{i,k}(\Theta )m_{k,i}(y,u)=0,\,i=1,\ldots,s
\end{eqnarray}
where $(c_{i,k}(\Theta ))_{1\leq k\leq n_i}$ are rational in $\Theta $, $c_{k,i}(\Theta ) \ne c_{l,i}(\Theta )$ if $k\ne l$, $(m_{k,i}(y,u))_{1\leq k\leq n_i}$ are differential polynomials relatively to $y$ and $u$ and $m_{0,i}(y,u)\ne 0$.\\
The computation of these polynomials does not require any consideration of the admissible parameter set $\mathcal{D}\subseteq \R^m$. Moreover, the number of input-output polynomials, detailed in~\cite{dj,VDJ}, is equal to the number of outputs of System~(\ref{eqpol}). \\
Polynomials $P_i$ are classically known as \textit{input-output polynomials} and in the literature, the sequence $(c _{i,j}(\Theta ))_{ 1\leq i\leq s ,\;1\leq j\leq n_i}$ is called the \textit{exhaustive summary}.  \\

In \cite{NOLCOS}, the global identifiability study was brought back to the injectivity study of the function 
$$\begin{array}{cccc} \label{deffunctionphi}
\phi : &  \mathcal{D} & \longrightarrow &\R^{N} \\
 & \Theta & \longmapsto  & (c_{i,j}(\Theta))_{1\leq i\leq s,1\leq j\leq n_i}\\
 &  &  & 
\end{array}$$
where $\displaystyle N=\sum_{i=1}^{s} n_i$.
In this paper, in order to take into account initial conditions or supplementary constraints on parameters, we restrict the domain of $\phi$ to a subset $\mathcal{C}$ of $\mathcal{D}$ (see. section \ref{DefC}).

The next proposition extends, to the case of relative identifiability, the link between global identifiabilty and result summary given in \cite{NOLCOS,VDJ}.

\begin{prp}\label{proprelidentif}
Suppose that for all $i=1,\ldots,s$, the functional determinants\footnote{In order to prove that determinants of Proposition~\ref{proprelidentif} are not identically equal to 0, it is sufficient to verify the linear independence of the $m_{k,i}(\bar y,\bar u)$ ($k=1,\ldots,s_i$). This can be done by computing the functional determinant given by the Wronskian (See~\cite{dj})
 \begin{equation}
\bigtriangleup P_{i}(\bar y,\bar u)=\left|
\begin{array}{ccc} \label{detfunct}
m_{1,i}(\bar y,\bar u)&\ldots &m_{s_i,i}(\bar y,\bar u)\\
m_{1,i}(\bar y,\bar u)^{(1)}&\ldots &m_{s_i,i}(\bar y,\bar u)^{(1)}\\
\vdots& &\vdots \\
m_{1,i}(\bar y,\bar u)^{(s_i-1)}&\ldots &m_{s_i,i}(\bar y,\bar u)^{(s_i-1)}\\
\end{array}
 \right| 
 \end{equation}
 and by verifying that $\bigtriangleup P_{i}(\bar y,\bar u)$ is not identically equal to zero.
} $$\bigtriangleup  P_i(y,u,\Theta )=\det(m _{k,i}(y,u),k=1,\ldots,n_i)$$ are not identically equal to zero. $\theta_{a_{r+1}}$ is relatively identifiable with respect to the eventually empty set of parameters $\{\theta_{a_{1}},\ldots,\theta_{a_{r}}\}$ if and only if, for any $\tilde \Theta =(\tilde  \theta_j)_{j=1,\ldots,m}\in \R^m$ and $\Theta =(\tilde  \theta_j)_{j=1,\ldots,m}\in \R^m$:
\begin{equation}\label{hypot1}
\left\{\begin{array}{ll}
\Theta \in \mathcal{C}& \\
\tilde \Theta \in \mathcal{C}&\\
\theta_{a_1}&=\tilde \theta_{a_1}\\
&\vdots \\
\theta_{a_r}&=\tilde \theta_{a_r}\\
\phi(\Theta)&=\phi(\tilde \Theta) 
\end{array}\right. \Rightarrow 
\theta_{a_{r+1}}=\tilde \theta_{a_{r+1}}
\end{equation} 
\end{prp}
\emph{Proof--} \textit{Sufficiency} Suppose that there exists $\tilde \Theta \in \mathcal{C}$ such that $\theta_{a_1}=\tilde \theta_{a_1},\ldots,
\theta_{a_r}=\tilde \theta_{a_r}$ and $y(t,\Theta)=y(t,\tilde \Theta)$ for all $t\in [0,T]$. The last equality implies that for all $j\in \N$, $y^{(j)}(t,\Theta)=y^{(j)}(t,\tilde \Theta)$ for all $t\in [0,T]$. By definition of the input-output polynomials $P_i(y,u,\Theta)$ and $P_i(y,u,\tilde \Theta)$, the difference $P_i(y,u,\Theta)-P_i(y,u,\tilde \Theta)=\sum_{k=1}^{n_i}(c_{i,k}(\Theta)-c_{i,k}(\tilde \Theta))m_{k,i}(y,u)$ is equal to 0. Since the functional determinant of the last difference is not identically equal to zero, we have $c_{i,k}(\Theta)-c_{i,k}(\tilde \Theta)=0$ which proves that the exhaustive summaries $\Phi(\Theta)$ and  $\Phi(\tilde\Theta)$ are equal. According to (\ref{hypot1}), $\theta_{a_{r+1}}=\tilde \theta_{a_{r+1}}$ and $\theta_{a_{r+1}}$ is relatively identifiable with respect to the set $\{\theta_{a_{1}},\ldots,\theta_{a_{r}}\}$.\\
\textit{Necessity} By contrapositive. Let us consider two vectors of parameters $\Theta=(\theta_1\, \ldots, \, \theta_{a_m} )$ and $\tilde \Theta=(\tilde\theta_1\, \ldots, \,\tilde \theta_{a_m})  $ such that $\theta_{a_1}=\tilde \theta_{a_1},\, \ldots, \, \theta_{a_r}=\tilde \theta_{a_r}$, $\theta_{a_{r+1}}\neq \tilde \theta_{a_{r+1}}$ and $c_{i,j}(\Theta) = c_{i,j}(\tilde\Theta) $. This implies that, for all $i\in \{1,\, \ldots,\, s\}$, $P_i(y,u,\Theta)=P_i(y,u,\tilde\Theta)$. Therefore, the input-output polynomials, and consequently their sets of solutions, are equal. According to Definition~\ref{defidentrelat}, $\theta_{a_{r+1}}$ is not relatively identifiable with respect to the set $\{\theta_{a_{1}},\ldots,\theta_{a_{r}}\}$.\\


 In the case of a global identifiability study, some approaches had been proposed in the literature to solve the real system $\phi(\Theta)=\phi(\tilde \Theta)$ \cite{NOLCOS,VDJ,DAISY2007,DAISY2010}. But these tools use Groebner basis algorithms which are not adapted for the resolution of general real algebraic systems. Moreover, the eventual inequalities satisfied by parameters can not be considered by these approaches whereas the identifiability result may depend on them (see Example \ref{subsectexemple1}).

In this paper, we propose an automatic procedure to prove Implication (\ref{hypot1}) including the parameter constraints. This procedure is based on semialgebraic tools. Some of these tools, already implemented in computer algebra systems have never been used to certify the (global, local and relative) identifiability of a model. 
 

From the system appearing in the left side of implication (\ref{hypot1}), it is possible to construct a system composed of polynomial equations and inequalities defining a semialgebraic set. Indeed, $\mathcal{C}$ is assumed to be a semialgebraic set and system $c_{i,j}(\Theta)=c_{i,j}(\tilde \Theta)$, $i=1,\ldots,s$, $j=1,\ldots,n_i$ is equivalent to an algebraic system by adding eventually non vanishing conditions on the denominators. These conditions can be added to $C(\Theta)$ and $C(\tilde \Theta)$. \\

Let set, for $ \Theta= ( \theta_{1}, \ldots,  \theta_{m})\in \mathbb{R}^m$, $\tilde \Theta= (\tilde \theta_{1}, \ldots, \tilde \theta_{m})\in \mathbb{R}^m$ and $0\leq r\leq n$, the following set of equations and inequalities:
 \begin{center}
 $
S_{\theta_{a_1}, \ldots, \theta_{a_r}}= C(\Theta)\ \cup\ C(\tilde{\Theta})\ \cup\ \{\theta_{a_i}= \tilde \theta_{a_i} \;\mid \;1\leq i\leq r  \}$ \end{center} $\hspace{50mm} \cup\ \{c_{i,j}(\Theta)=c_{i,j}(\tilde\Theta)\;\mid \;1\leq i\leq s,\; 1\leq j\leq n_i \}\,.
$\\

In order to use semialgebraic tools for proving implication \ref{hypot1}, the following corollary is based on the classical tool of computer algebra techniques called the Rabinowitsch trick (See~\cite{Basu:2006:ARA:1197095,Rabinowitsch}): to show that an algebraic system has a solution with $u\neq 0$, where $u$ is one of its indeterminates, this trick consists in adding the equation $uv-1=0$ to the system, where $v$ is a new indeterminate, and to prove that the new system has a solution. By this way, computer algebra packages testing the emptiness of semialgebraic sets can be used such as \verb HasRealSolution  of the Maple package Raglib ~\cite{RAGLib} for example. 

\begin{cor} \label{proptestp} Let $(\theta_{1}, \ldots, \theta_{m})$ and $(\tilde \theta_{1}, \ldots, \tilde \theta_{m})$ in $\mathcal{C}$. The following conditions are equivalent:
\begin{enumerate}
 \item the parameter $\theta_{a_{r+1}}$ of System $\Gamma^\Theta$ is relatively identifiable with respect to the eventually empty set of parameters $\{\theta_{a_1}, \ldots, \theta_{a_r}\}$;
 \item the system $S_{\theta_{a_1}, \ldots, \theta_{a_r}} \cup\,  \left\lbrace v\left(  \theta_{a_{r+1}}-\tilde \theta_{a_{r+1}}\right)-1= 0\right\rbrace $, where $v$ is a new variable, has no real solution $(\theta_{a_1}, \ldots, \theta_{a_m},\tilde \theta_{a_1}, \ldots, \tilde \theta_{a_m},v)\in \mathbb{R}^{2m+1}$. 
\end{enumerate}
\end{cor}

\emph{Proof--} \textit{$1.\, \Rightarrow\, 2.$} By contrapositive. If system $S_{\theta_{a_1}, \ldots, \theta_{a_r}} \cup\,  \left\lbrace v\left(  \theta_{a_{r+1}}-\tilde \theta_{a_{r+1}}\right)-1= 0\right\rbrace $, where $v$ is a new variable, has a real solution, then this solution satisfies  $\theta_{a_{r+1}}-\tilde \theta_{a_{r+1}}\neq 0$. According to Proposition~\ref{proprelidentif}, the parameter $\theta_{a_{r+1}}$ is not relatively identifiable with respect to $\{ \theta_{a_1}, \ldots, \theta_{a_r}\}$.\\

\textit{$2.\, \Rightarrow\, 1.$} By contrapositive. If $\theta_{a_{r+1}}$ of System $S_{\theta_{a_1}, \ldots, \theta_{a_r}}$ is not relatively  identifiable with respect to $\{\theta_{a_1}, \ldots, \theta_{a_r}\}$, there exist two solutions $(\theta_{a_1}, \ldots, \theta_{a_m})$ and $(\theta_{a_1}, \ldots, \theta_{a_r}, \tilde \theta_{a_{r+1}}, \ldots, \tilde \theta_{a_m})$  of this system such that  $\theta_{a_{r+1}} \neq \tilde \theta_{a_{r+1}}$ (See Proposition~\ref{proprelidentif}). This inequality is equivalent to the existence of a real $v$ such that $v \left( \theta_{a_{r+1}}-\tilde \theta_{a_{r+1}}\right)-1=0$. Consequently, the $2m+1$ t-uplet of real $(\theta_{a_1}, \ldots, \theta_{a_m},\theta_{a_1}, \ldots, \theta_{a_r}, \tilde \theta_{a_{r+1}}, \ldots, \tilde \theta_{a_m},v)$ is a solution of system $S_{\theta_{a_1}, \ldots, \theta_{a_r}} \cup\,  \left\lbrace v\left(  \theta_{a_{r+1}}-\tilde \theta_{a_{r+1}}\right)-1= 0\right\rbrace $.\\


\section{Relative identifiability algorithm}\label{section2}

In this section, we present our algorithm for the relative identifiability study of parameters of System $\Gamma^\Theta$.

\subsection{Useful concepts}

In order to describe our algorithm and its outputs, we recall the following notations and definitions of computer science.

\begin{notation} 
$[\theta_{a_1}, \ldots, \theta_{a_{r}}]$ is the list\footnote{In computer science, a list is a finite \textbf{sequence} of elements of a set.} of $r$ distinct elements ($1\leq r \leq m$)  taken among $\{\theta_1,\ldots,\theta_m\}$.
\end{notation}
For the construction of parameter lists by our algorithm, let us recall the following classical definitions.

\begin{dfn} 
\begin{enumerate}
 \item Let $[ \theta_{a_1}, \ldots, \theta_{a_{r}}]$ be a list of $r$ distinct elements of $\{\theta_{1}, \ldots, \theta_{m}\}$. \\The empty list and any list $[ \theta_{a_1}, \ldots, \theta_{a_{i}}]$ ($1\leq  i \leq r-1$) are called prefix of $[ \theta_{a_1}, \ldots, \theta_{a_{r}}]$.
 \item Let $l_1$ and $l_2$ be two lists. The concatenated list $l_1$ \verb cat  $l_2$ is the list obtained by joining the two lists $l_1$ and $l_2$ end-to-end.
 
\end{enumerate}
\end{dfn}
In order to distinguish identifiable parameters to non identifiable ones in a list, we introduce the notation hereafter. 
	
\begin{notation}\label{notationbarre}
Let $l =  [ \theta_{a_1}, \ldots, \theta_{a_{r}}]$ be a list of $r$ elements taken among $\{ \theta_{1}, \ldots, \theta_{m}\} $.

If the parameter $ \theta_{a_1}$ is not identifiable, it will be written $ {{\cancel{\theta}}_{a_1}}$ in the list~$l$.

More generally, if $\theta_{a_i}$ $( 2 \leq i \leq r)$  is not relatively identifiable with respect to $\{\theta_{a_1}, \ldots, \theta_{a_{i-1}}\}$, we write it $ {{\cancel{\theta}}_{a_i}}$ in $l$.
\end{notation}

From now on, in order to lighten the text, a parameter will be said relatively identifiable with respect to a list of parameters if it is relatively identifiable with respect to the corresponding set of parameters. \\

\subsection{Relative identifiability study}\label{subsecalgo}

A way to obtain an exploitable output for an algorithm is to use lists of parameters rather than sets and naturally to consider a tree traversal algorithm in order to avoid useless emptiness tests of semialgebraic sets and also redundant outputs. Such an algorithm consists in constructing all the possible lists of $\theta_{1}, \ldots, \theta_{{m}}$ and in indicating non identifiable parameters of these lists. This approach leads to the following definition which is the output of our algorithm.

\begin{dfn} \label{defindtree} 
The identifiability tree $\mathcal{T}$ of $\Gamma^\Theta$ is the set of all possible lists of parameters taken among $\{ \theta_{1}, \ldots, \theta_{m}\}$ such that, for all list $b \in \mathcal{T}$ and any prefix $p$ of $b$, $p$ is followed in $b$ by an identifiable parameter relatively to $p$ if there exists one. 
\end{dfn}

Note that if lists of $\mathcal{T}$ have a common prefix, tests performed for this prefix will be computed only once. Moreover, in order to avoid some useless tests, the following observations are used in our algorithm.\\

\textbf{Two important observations to improve the efficiency of the identifiability tree algorithm.} 

\begin{enumerate}
 \item \textbf{Observation 1.}  \textit{(Consequence of Proposition~\ref{propdefidentrelat}) 
If a prefix $p$ of a list of $\mathcal{T}$ is followed by successive identifiable parameters, these identifiable parameters can be permuted in this list since it refers to the same information: each of them are relatively identifiable with respect to $p$.} 

 \item \textbf{Observation 2.} \textit{If lists completing a given prefix $p_1$ into elements of  $\mathcal{T}$ are known, these lists complete any permutation of $p_1$ into elements of~$\mathcal{T}$.} Indeed, the relative identifiability of a parameter relatively to a prefix $p$ depends only on the set of parameters appearing in $p$.
  
\end{enumerate}


This last observation leads naturally to a pre-order tree traversal (see~\cite{bg}) in order to complete prefixes with lists of $\mathcal{T}$ already computed. By this way, some emptiness tests of semialgebraic sets can be avoided which decrease the global cost of the computation of $\mathcal{T}$ essentially due to these tests. 

\subsection{Algorithm for studying the relative identifiability}

In this section, we describe our recursive algorithm called \verb IdentifiabilityTree  which returns the identifiability tree $\mathcal{T}$. 
Identifiability tests are realized by applying Corollary~\ref{proptestp}.
 \\


This algorithm takes as input any prefix $p$ of a list of $\mathcal{T}$ and computes all the lists of $\mathcal{T}$ admitting $p$ as prefix. Lists of $\mathcal{T}$ produced by the algorithm are stored in the set $\mathcal{T}'$. In particular, the entire identifiability tree is returned by the call \verb IdentifiabilityTree $(p)$, with $p=[]$ and $\mathcal{T}'$ initialized to $\emptyset$. \\

During any recursive call, the input $p$ is completed first with all the identifiable parameters with respect to $p$ in any order (see~Observation 1) and, after, by each of the non identifiable parameters. Generated prefixes are then completed by recursive calls of the algorithm. When it is possible, Observation 2 is used to complete prefixes using $\mathcal{T}'$ avoiding useless semialgebraic emptiness tests.\\ 

Let us summarize the different steps of the algorithm applied to a prefix~$p$. \\

\noindent\verb IdentifiabilityTree $(p)$ \hrulefill\\

\noindent 1. While there exists some identifiable parameters relatively to $p$ do\\
\phantom{AAAA} 1.1 Complete $p$ with an identifiable parameter $\theta$ ;\\
\phantom{AAAA} 1.2 Check whether Observation 2. can be applied to complete $p$ cat $[\theta]$ into new lists of the identifiability tree. If it is the case, add these new lists to $\mathcal{T}'$, stop this While loop and do not proceed to step 2.  \\
 
\noindent 2. If there exists some non identifiable parameters relatively to $p$, for each of these parameters, do \\
\phantom{AAAA} 2.1 Complete $p$ with a non identifiable parameter $ \cancel{\theta} $ ;\\
\phantom{AAAA} 2.2 If Observation 2. can be applied to complete $p$ cat $[\cancel{\theta}]$ into new lists of the identifiability tree, add these new lists to $\mathcal{T}'$. Otherwise, do a recursive call of the algorithm with $p$ cat $[\cancel{\theta}]$ as input.

\noindent \hrulefill
\bigskip

The following proposition gives the complexity of our algorithm.
\begin{prp} [Complexity] \label{complexity}
The number of emptiness tests of semialgebraic sets performed by our algorithm is bounded by $(2m-\nu+2) 2^{\nu-1}$ where $m$ (resp.~$\nu$) is the number of parameters (resp. of non identifiable parameters) of System~$S$.\\

\end{prp}

\begin{proof}
Let $\mathcal{T}_r$ the set of prefixes of minimal length of lists of $\mathcal{T}$ admitting exactly $r$ non identifiable parameters $(0\leq r \leq \nu)$. Let $\mathcal{C}_r$ be the subset of $\mathcal{T}_r$ composed of prefixes for which the non identifiabilty of parameters require exactly $r$ emptiness tests of semialgebraic sets.

For each prefix of $\mathcal{C}_r$ constructed by the algorithm, at least $r!$ prefixes of $\mathcal{T}_r$ are completed into lists of $\mathcal{T}$ using Observation 2 that is without any tests. Since there is at most ${\nu(\nu-1) \cdots (\nu-(r-1))}$ prefixes, up to a permutation of non identifiable parameters, $\mathcal{C}_r$ contains at most $ {\nu \choose {r}}$ prefixes.

This implies that the number of tests needed for the last non identifiable parameters appearing in the lists of  $\mathcal{C}_r$ is bounded by ${\nu \choose r}$ and, consequently, that the number of tests performed by our algorithm for the determination of the non identifiable parameters is bounded by $\sum_{r=0}^{\nu}{\nu \choose {r}}  =  2^{\nu}\,.$

Moreover, for each prefix $p$ of $\mathcal{T}_r$, there are $m-r$ identifiable parameters in the list of $\mathcal{T}$ corresponding to $p$. Consequently, at most $(m-r ){\nu \choose {r}}$ tests are performed for the identifiable parameters of lists of $\mathcal{T}$ completing lists of  $\mathcal{T}_r$. This implies that  the number of tests realized by our algorithm for the identifiable parameters is bounded by $$\sum_{r=0}^{\nu}(m-r){\nu \choose {r}}  = (2 m - \nu  )  2^{\nu-1}\,.$$

Therefore, the number of test realized by our algorithm for computing the identifiability tree
 is bounded by $(2 m - \nu + 2 ) 2^{\nu-1}$. 
 \end{proof}

%
%
%
%
%

\subsection{Steps of the method} \label{SubSectMethode}
The step for studying relative identifiability of System~(\ref{saresoudre}) are summed up below:\\

\noindent \textbf{Step 1. } Determination of the set of conditions $C(\Theta)$, including eventually the initial conditions.\\
\noindent \textbf{Step 2. } Elimination of unobservable state variables using the Rosenfeld-Groebner algorithm in order to obtain the input output polynomials of the system.\\
\noindent \textbf{Step 3. } Construction of the exhaustive summary.\\
\noindent \textbf{Step 4. } Verification of the assumption of Proposition~\ref{proprelidentif}.\\
\noindent \textbf{Step 5. } Computation of the relative identifiability tree by our algorithm.\\

\section{Examples}\label{secexemples}

The examples presented in this section had been treated by an implementation of our algorithm in the computer algebra system Maple. In order to test relative identifiability using Corollary \ref{proptestp}, emptiness tests of semialgebraic sets had been realized with the function \verb IsEmpty  of the Maple package \verb SemiAlgebraicSetTools . \\

The first example revisits the identifiability of the model introduced by \cite{Holmberg1982} and describes the microbial growth in a batch reactor. The identifiability tree study shows that, for this example, initial conditions do not change the identifiability of the model contrary to parameter constraints. Moreover, with a supplementary observation, we deduce the identifiability of the model from the output of our algorithm.\\

The second example of this section, dealing with the epidemiological model taken from~\cite{Moulay2010}, is more complex in term of number of parameters. Despite the cost of the computation and the size of the exhaustive summary and implicitly the number of algebraic polynomials involved in the input-output polynomial, our algorithm provides lists of parameters that should be estimated to turn this model into an identifiable one.

\subsection{Microbial growth in a batch reactor}
\label{subsectexemple1}
This model is known as being unidentifiable from an impulse input experiment (see~\cite{chappell}). The dynamic model for the growth process is governed by the two following equations: 
\begin{equation}
\Gamma^\Theta~\left\{
\begin{array}{l}\label{eqexemple1a}
\dot x(t)= \dfrac{\mu \, s(t) \, x(t)}{K_S + s(t)} -m Y \, x(t),\\
\dot s(t)=\dfrac{-\mu \, s(t)\, x(t)}{Y\, (K_S+s(t))}
\end{array}\right.
\end{equation}
with initial conditions 
\begin{equation}
\begin{array}{l}
x(0^+)=x_0,\\
s(0^+)=s_0
\end{array}
\end{equation}

\noindent and where $x$ is the concentration of microorganisms; $s$, the concentration of growth-limiting substrate, $\mu$ the maximum velocity of the reaction; $K_S$ the Michaelis-Menten constant and $Y$ the yield coefficient. Due to the nature of the parameters, $\mu$, $K_S$, $m$ and $Y$ are supposed to be positive reals.\\ 

Let us consider that the output is $y=x$. We now apply the different steps of the method proposed at Section~\ref{SubSectMethode}.\\

\noindent \textbf{Step 1.} The significations of the parameters lead to the set of conditions $ C(\Theta)=\{\mu > 0,  K_S >0,  m>0 , Y>0\}$.\\
\noindent \textbf{Step 2.} The Rosenfeld-Groebner algorithm applied to System \ref{eqexemple1a} with the elimination order induced by $ y \prec x \prec s $  returns the following input-output polynomial:
\begin{equation}
\begin{array}{l}
(Y^3m^3-2Y^2m^2\mu +Ym\mu^2)y^3+ (3Y^2m^2-4Ym\mu +\mu^2)y^2\dot y\qquad\qquad\\\qquad\qquad\qquad\qquad\qquad\qquad+K_SY\mu (y\ddot y -\dot y^2)+(3Ym-2\mu )y\dot y^2+\dot y^3\,.
\end{array}
\end{equation}
\noindent \textbf{Step 3. } From this differential polynomial, we obtain the exhaustive summary $$( \mu KSY,\; -3Ym+2\mu,\;  3Y^2m^2-4Ym\mu+\mu^2,\;  Y^3m^3-2Y^2m^2\mu+Ym\mu^2)$$                                                                                                                                                          
which is  also the image of $\Theta = (\mu, K_S, m,Y)$ by $\phi$ (See Section~\ref{deffunctionphi}).\\
\noindent \textbf{Step 4. } To verify that assumption of Proposition~\ref{proprelidentif} is satisfied, computer algebra system is used to verify that the functional determinant~(\ref{detfunct}) is not identically equal to 0.\\
\noindent \textbf{Step 5. } The relative identifiability tree can then be computed by Algorithm \verb IdentifiabilityTree . The algorithm runs through the full tree of Figure~\ref{Tree1}.

\begin{figure}[!h]
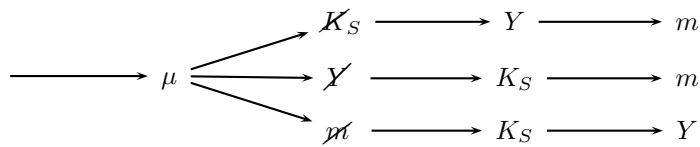

\centering
\psset{arrows=->}
\pstree[treemode=R,levelsep=15ex,treesep =4mm, nodesep=5pt]{\TR{} }{
        \pstree{ \TR{${\mu}$} }{\pstree{ \TR{${\cancel{K}_S}$} }{\pstree{ \TR{${Y}$} }{\TR{${m}$}}}
				\pstree{ \TR{${\cancel{Y}}$  }}{\pstree{ \TR{${K_S}$} }{\TR{${m}$}}}
				\pstree{ \TR{${\cancel{m}}$  }}{\pstree{ \TR{${K_S}$} }{\TR{${Y}$}}}}
                    }
\caption{Tree transversal realized by Algorithm IdentifiabilityTree } \label{Tree1}	
\end{figure}

The output of the algorithm is $\mathcal{T} =  \{[\mu, \cancel{K}_S, Y, m], [\mu, \cancel{Y}, K_S, m], [\mu, \cancel{m}, K_S, Y]\}$. \\

These lists show that $\mu$ is identifiable and that none of the three other parameters are identifiable. Remark that if one parameter among $ K_S$, $m$ or $Y$ is known then the system is identifiable.\\

\begin{remark} 

\begin{enumerate}
 \item In this example, the role of the constraints imposed on the parameters plays a crucial role for establishing the identifiability of System~\ref{eqexemple1a}. Indeed, if constraints $C(\Theta)$ are not considered, i.e. if $C(\Theta)= \emptyset $, we obtain \begin{footnotesize}$$\mathcal{T} = \{ [\mu, \cancel{K_S}, \cancel{Y}, \cancel{m}], [\mu, \cancel{K_S}, \cancel{m}, \cancel{Y}], [\mu, \cancel{Y}, \cancel{K_S}, \cancel{m}], [\mu, \cancel{Y}, \cancel{m}, \cancel{K_S}], [\mu, \cancel{m}, \cancel{K_S}, \cancel{Y}], [\mu, \cancel{m}, \cancel{Y}, \cancel{K_S}]\}\,.$$                                                                                                                                                                                                                                                                                                                         	     \end{footnotesize}
Contrary to the previous situation, the tree parameters have to be identified to turn the system into an identifiable one. 

 \item In order to evaluate the role of initial conditions, we introduce the new unknown parameters $x_0$, $s_0$, $sp_0$, $xp_0 $ corresponding respectively to $x(0)$, $\dot x(0) $, $s(0)$, $\dot s(0)$ and we add to $C(\Theta)$ equations obtained from System~(\ref{eqexemple1a}) evaluated at $t=0$. 
 
Computation of $\mathcal{T}$ shows that in any list, the parameter $\mu$ is identifiable and
										\begin{itemize}                                        
                                         \item either two identifiable parameters among $x_0$, $s_0$, $sp_0$, $xp_0 $ are not identifiable and exactly one  parameter among $K_s$, $Y$ and $ m$ is not identifiable;
                                         \item or three parameters among $x_0$, $s_0$, $sp_0$, $xp_0 $ are not identifiable.
                                        \end{itemize}
We can conclude, from the second case that if three initial conditions are known, $K_s$, $Y$ and $ m$ are identifiable.

 \item In the same way, without considering the initial conditions, when outputs $x$ and $s$ are supposed to be observed, the identifiability tree is reduced to $ \{[K_S, Y, m, \mu]\}$. This means that the model is identifiable.
\end{enumerate}

\end{remark}

\subsection{Chikungunya model} \label{section_Chikungunya}
The following model $\Gamma^\Theta$ is an epidemiological model describing the propagation of the Chikungunya disease to human population proposed in~\cite{Moulay2010}. In \cite{Shousheng2014}, an identifiability study had been done on this model assuming that some parameters are known. In this section, we assume that none of the parameters are known and we propose to find lists of key parameters which turn  $\Gamma^\Theta$ into an identifiable model.\\

$\Gamma^\Theta~~~\begin{cases}\Gamma^\Theta_1
\begin{cases}
E'(t)=b A(t)\left( 1-\dfrac{E(t)}{k_E}\right) -(s+d)E(t)  \\
L'(t)=s E(t)\left( 1-\dfrac{L(t)}{k_L}\right) -(s_L+d_L)L(t)  \\
A'(t)=s_L L(t)-d_m A(t) \\
\end{cases}
  \\
~  \\\Gamma^\Theta_2
\begin{cases}
S_H'(t)=-\left(b_H+\beta_H I_m(t)\right)S_H(t)+b_H    \\
I_H'(t)=\beta_H I_m(t) S_H(t)-(\gamma+b_H)I_H(t)    \\
I_m'(t)=-\left( s_L\dfrac{L(t)}{A(t)}+\beta_m I_H(t)\right) I_m(t)+\beta_m I_H(t)   \\
\end{cases}
\end{cases}$\\

Let us explain briefly this system. Equations appearing in $\Gamma^\Theta_1$ correspond to the three biological steps in the life cycle of the mosquito transmitting the disease : $E$ is the number of eggs, $L$ the number of pupae and $E$ the number of adult females. Assuming that the number of larvae can be counted weekly by biologists, this variable is considered as one of the measured variable of the model. In System $\Gamma^\Theta_2$, $S_H$ corresponds to the number of humans susceptible to be infected, $I_H$ to the number of infected humans and $I_M$ to infected mosquitoes. In $\Gamma^\Theta_2$, $S_H$ and $I_H$ can be supposed to be measured variables. 
As in~\cite{Shousheng2014}, let set $y_1= L$, $y_2= S_H$ and $y_3 = I_H$.\\

The vector of unknown parameters is $\Theta=(k_E, k_L, b_H, \beta_H , \beta_m , d_L, d_m, \gamma , b, d, s, s_L )  $ where
\begin{itemize}
 \item $b$ is the intrinsic rate of eggs, $s$ (resp. $s_L$) is the
transfer rate from $E$ to $L$ (resp. from $L$ to $A$);
 \item $k_E$ (resp. $k_L$ ) is the carrying capacity of $E$
(resp. carrying capacity of $L$);
 \item $d$, $d_L$ and $d_M$ are the rates of natural deaths for
eggs, larvae and adults;
 \item $b_H$ the human birth;
 \item $\gamma$ is the transfer rate between infected humans
and recovered humans;
 \item $\beta_H$ (resp. $\beta_M$) is the infectious contact rate between susceptible humans and vectors (resp. susceptible mosquitoes and humans).
\end{itemize}

Naturally, real parameters of System~$\Gamma^\Theta$ are supposed to be positive which leads to set  $C(\Theta)=\{b_H>0, \,\beta_H>0, \,\gamma>0, \,s_L>0, \,\beta_M>0, \,d>0, \,d_L>0, \,d_M>0, \,b>0, \,s>0, \,k_E>0, \,k_L>0, \,b_H>0\}\,.$

In order to apply the Rosenfeld-Groebner algorithm to system $\Gamma^\Theta$, the last equation of $\Gamma_2^{\Theta}$ is multiplied by function $A$. This algorithm used with the elimination order induced by $$[y_1, y_2,y_3]\prec [S_H,  I_H, L,  I_M,E,  A ]  $$ returns three differential equations linking unknown parameters and functions $y_1$, $y_2$, $y_3$. One of these equations does not have a constant coefficient $c_{ij}$ and therefore its coefficients can be estimated uniquely up to a multiplicative constant from experimental values (see~\cite{VDJ}). This equation have to be divided by one of these coefficients: $k_E k_L$ is used in this application. From these equations, we obtain an exhaustive summary containing the 694 coefficients of the latter three differential equations. This exhaustive summary can be reduced to 212 coefficients not differing by a multiplicative constant. 
Our algorithm applied to the system obtained from $C(\Theta)$ and function $\phi$ returns the following set of \mbox{lists
\footnote{Despite the size of the semialgebraic system, the time needed for the computation of the identifiability tree is 880 s with Intel(R) Core(TM) i7 2.5 GHz processor with 8 GO of RAM.}:}
\begin{small} 

\begin{center}

\begin{tabular}[c]{cc}

$\{  [k_l,\,b_H,\,\beta_H,\,\beta_M,\,d_m,\,g,\,\cancel{k}_e,\,d,\,s,\,\cancel{b},\,d_l,\,s_l] ,$ & $[k_l,\,b_H,\,\beta_H,\,\beta_M,\,d_m,\,g,\,\cancel{k}_e,\,d,\,s,\,\cancel{d}_l,\,b,\,s_l] ,\,$  \\
  $[k_l,\,b_H,\,\beta_H,\,\beta_M,\,d_m,\,g,\,\cancel{k}_e,\,d,\,s,\,\cancel{s}_l,\,b,\,d_l] ,$ & $[k_l,\,b_H,\,\beta_H,\,\beta_M,\,d_m,\,g,\,\cancel{b},\,\cancel{k}_e,\,d,\,d_l,\,s,\,s_l] ,$  \\
$[k_l,\,b_H,\,\beta_H,\,\beta_M,\,d_m,\,g,\,\cancel{b},\,\cancel{d},\,k_e,\,d_l,\,s,\,s_l] ,$ & $[k_l,\,b_H,\,\beta_H,\,\beta_M,\,d_m,\,g,\,\cancel{b},\,\cancel{d}_l,\,k_e,\,d,\,s,\,s_l] ,$   \\

     $[k_l,\,b_H,\,\beta_H,\,\beta_M,\,d_m,\,g,\,\cancel{b},\,\cancel{s},\,k_e,\,d,\,d_l,\,s_l] ,$ & $[k_l,\,b_H,\,\beta_H,\,\beta_M,\,d_m,\,g,\,\cancel{b},\,\cancel{s}_l,\,k_e,\,d,\,d_l,\,s] ,$   \\
 $[k_l,\,b_H,\,\beta_H,\,\beta_M,\,d_m,\,g,\,\cancel{d},\,k_e,\,s,\,\cancel{b},\,d_l,\,s_l] ,$ &    $[k_l,\,b_H,\,\beta_H,\,\beta_M,\,d_m,\,g,\,\cancel{d},\,k_e,\,s,\,\cancel{d}_l,\,b,\,s_l] ,$   \\
  $[k_l,\,b_H,\,\beta_H,\,\beta_M,\,d_m,\,g,\,\cancel{d},\,k_e,\,s,\,\cancel{s}_l,\,b,\,d_l] ,$ & $[k_l,\,b_H,\,\beta_H,\,\beta_M,\,d_m,\,g,\,\cancel{d}_l,\,sl,\,\cancel{k}_e,\,b,\,d,\,s] ,$   \\

     $[k_l,\,b_H,\,\beta_H,\,\beta_M,\,d_m,\,g,\,\cancel{d}_l,\,s_l,\,\cancel{b},\,k_e,\,d,\,s] ,$ & $[k_l,\,b_H,\,\beta_H,\,\beta_M,\,d_m,\,g,\,\cancel{d}_l,\,s_l,\,\cancel{d},\,k_e,\,b,\,s] ,$  \\
 $[k_l,\,b_H,\,\beta_H,\,\beta_M,\,d_m,\,g,\,\cancel{d}_l,\,s_l,\,\cancel{s},\,k_e,\,b,\,d] ,$ &    $[k_l,\,b_H,\,\beta_H,\,\beta_M,\,d_m,\,g,\,\cancel{s},\,k_e,\,d,\,\cancel{b},\,d_l,\,s_l] ,$  \\
 $[k_l,\,b_H,\,\beta_H,\,\beta_M,\,d_m,\,g,\,\cancel{s},\,k_e,\,d,\,\cancel{d}_l,\,b,\,s_l] ,$ & $[k_l,\,b_H,\,\beta_H,\,\beta_M,\,d_m,\,g,\,\cancel{s},\,k_e,\,d,\,\cancel{s}_l,\,b,\,d_l] ,$  \\

     $[k_l,\,b_H,\,\beta_H,\,\beta_M,\,d_m,\,g,\,\cancel{s}_l,\,d_l,\,\cancel{k}_e,\,b,\,d,\,s] ,$ & $[k_l,\,b_H,\,\beta_H,\,\beta_M,\,d_m,\,g,\,\cancel{s}_l,\,d_l,\,\cancel{b},\,k_e,\,d,\,s] ,$\\
 $[k_l,\,b_H,\,\beta_H,\,\beta_M,\,d_m,\,g,\,\cancel{s}_l,\,d_l,\,\cancel{d},\,k_e,\,b,\,s] ,$ &     $[k_l,\,b_H,\,\beta_H,\,\beta_M,\,d_m,\,g,\,\cancel{s}_l,\,d_l,\,\cancel{s},\,k_e,\,b,\,d]    \} $ \\  
\end{tabular}
\end{center}
 \end{small}
   
This output shows that $k_l, b_H, \beta_H, \beta_M, d_m, g $ are identifiable. In order to identify the whole system, this result indicates which pairs of other parameters have to be estimated.

\section{Conclusion}\label{conclusion}
The introduction of the \textit{relative identifiability} definition permits to elaborate strategies for studying parameters identifiability of nonlinear models. The proposed algorithm does not require any experimental values of inputs and outputs and gives a mean for studying \textit{a priori} relative identifiability of differential models. This method enables the user to take into account inequalities satisfied by parameters and initial conditions to study their identifiability.\\This method is composed of two steps. The first one consists in the determination of algebraic relations between the outputs, the inputs and the unknown parameters by a classical elimination method used in differential algebra. The second step consists in using the algorithm proposed in this paper in order to determine the identifiability of any parameter relatively to any set of other parameters by testing emptiness of semialgebraic sets.\\ The output of this algorithm can be used, in particular, to determine which parameters must be known to make the system identifiable.

\bibliographystyle{plain} 

\end{document}